\documentclass[a4paper,12pt]{article}
\usepackage{amsmath,amsthm,amssymb,amsfonts}
\usepackage[latin1]{inputenc}
\usepackage{enumerate,enumitem}
\usepackage{graphics,graphicx,epsfig,psfrag}
\evensidemargin = \oddsidemargin
\advance\textwidth by -2in
\theoremstyle{plain}
\newtheorem{theorem}{Theorem}
\newtheorem{proposition}{Proposition}
\newtheorem{lemma}[proposition]{Lemma}
\theoremstyle{remark}

\def\N{{\mathbb N}}
\def\Q{{\mathbb Q}}
\def\R{{\mathbb R}}
\def\Z{{\mathbb Z}}
\def\CC{{\mathcal C}}
\def\DD{{\mathcal D}}
\def\SS{{\mathcal S}}
\def\ZZ{{\mathcal Z}}
\def\o{{\mathrm o}}
\def\id{{\mathrm{id}}}
\def\ly{\fontencoding{U}\fontfamily{lasy}\fontseries{m}\fontshape{n}\selectfont}
\def\guil#1{\leavevmode\hbox{{\ly(\kern-0.20em(\kern+0.20em}}\nobreak{}\,#1\,%
  \nobreak\hbox{{\ly\kern+0.20em)\kern-0.20em)}}}
\def\up{\textup}
\def\from{\colon} 
\def\bars#1{\lvert #1 \rvert} 
\def\bigbars#1{\bigl\lvert #1 \bigr\rvert} 
\def\Bigbars#1{\Bigl\lvert #1 \Bigr\rvert} 
\def\lrbars#1{\left\lvert #1 \right\rvert} 
\def\Bars#1{\lVert #1 \rVert} 
\def\bigBars#1{\bigl\lVert #1 \bigr\rVert} 
\def\BigBars#1{\Bigl\lVert #1 \Bigr\rVert} 
\def\lrBars#1{\left\lVert #1 \right\rVert} 
 
\def\th{^{\text{th}}}
\def\o{{\mathrm o}}
\def\O{{\mathrm O}}

\begin{document}
\title{On the centralizer of diffeomorphisms of the half-line}
\author{Hélène \textsc{Eynard}\thanks{UMPA (ENS Lyon),
heynardb@umpa.ens-lyon.fr}}
\maketitle

Let $f$ be a smooth diffeomorphism of the closed half-line $\R_+$ with a single 
fixed point at the origin. In this article, we study the centralizer of $f$ in 
the group $\DD^r$ of $\CC^r$ diffeomorphisms of $\R_+$, $1 \le r \le \infty$, 
that is, the (closed) subgroup $\ZZ^r_f$ of $\DD^r$ made up of all diffeomorphisms
commuting with $f$. The first things to observe are that $\ZZ^r$ decreases with
$r$,  contains the infinite cyclic subgroup generated by $f$ and is quite small.
Indeed, for $r=1$, well-known theorems by G.~Szekeres and N.~Kopell \cite{Sz,Ko}
show that $\ZZ^1_f$ is always a one-parameter subgroup of $\DD^1$. For $r\ge2$, 
the situation is subtler, and for instance both of the limit cases permitted by the
inclusions
\begin{equation*}
   \Z \cong \{ f^n,\ n \in \Z \} \quad\subset\quad \ZZ^r_f \quad\subset\quad \ZZ^1_f \cong \R 
\end{equation*}
can occur. According to F. Takens' work \cite{Ta}, if $f$ is not infinitely
tangent to the identity at $0$ then $\ZZ^1_f$ consists of smooth diffeomorphisms
and therefore coincides with $\ZZ^\infty_f$. On the other hand, in \cite{Se}, 
F.~Sergeraert builds a diffeomorphism~$f$ whose centralizer $\ZZ^2_f$ is 
strictly contained in $\ZZ^1_f$, and one can actually check~\cite{Ey} that, in
this example, $\ZZ^2_f$ reduces to the group spanned by~$f$ ---~and is hence as 
small as possible. The following result says that there exist intermediate 
situations:

\begin{theorem} \label{t:centralizer}
There exists a smooth diffeomorphism $f$ of $\R_+$ with a single fixed point at
the origin, whose centralizer $\ZZ^r_f$, for $2 \le r \le \infty$, is a
proper, dense and uncountable subgroup of the one-parameter group $\ZZ^1_f$.
\end{theorem}

This theorem follows from the proposition below, where $f$ is the flow at time 
one of the vector field~$\xi$ coming out:

\begin{proposition} \label{p:field} 
There exists a complete $\CC^1$ vector field $\xi$ on $\R_+$, vanishing only at
$0$, whose flow $f^t$ at time $t$ is not $\CC^2$ at $0$ for $t =  1 / 2$ but
is smooth on $\R_+$ for all $t \in \Z \oplus \sum_{\tau \in K} \tau\Z$, where 
$K \subset \R \setminus \Q$ is a Cantor set.
\end{proposition}

\medskip

A natural question to ask now is whether diffeomorphisms $f$ whose centralizer 
$\ZZ^r_f$, $r \ge 2$, is neither the one-parameter group generated by $f$ 
(namely, $\ZZ^1_f \cong \R$) nor the discrete group spanned by $f$ (that is, 
$\{ f^n, n \in \Z \} \cong \Z$), are very peculiar or not. At the end of the
paper, Theorem \ref{t:general} gives a partial answer to this question: every 
diffeomorphism of $\R_+$ which satisfies a certain oscillation condition and
belongs to a smooth flow (with the usual hypotheses on the unique fixed point) can 
be approximated in a suitable sense by diffeomorphisms $f$ whose centralizer 
$\ZZ^r_f$ is as in Theorem \ref{t:centralizer}. The proof of this second theorem
is very similar to that of the first but involves more technicalities. For this 
reason, we discuss the weaker statement in priority. 

It would also be interesting to know if the centralizer $\ZZ^r_f$, when it is a
proper subgroup of $\R \cong \ZZ^1_f$, can have positive Lebesgue measure, or
even contain any Diophantine number. It turns out \cite{Ey} that the Cantor set 
we construct in our proof of Proposition \ref{p:field} consists of Liouville
numbers, and hence has measure zero.\medskip

\textbf{Acknowledgements} I am extremely grateful to Sylvain Crovisier for explaining the method of
approximation by conjugation to me and suggesting that it could be used in this
work to preserve and control the desired smoothness of the limit flow. More
generally I am deeply thankful for his continued interest in my progress and
his useful comments on this article. I would 
also like to thank Jean-Christophe Yoccoz for sharing his insight on the subject
with me and encouraging me to work on this 
particular question. These two interactions were possible thanks to the
financial support of the Agence Nationale de la Recherche (through the ``Symplexe'' project). Last but not least, this work would not have been possible
without the considerable help of Emmanuel Giroux, who dedicated
much of his time and energy to me through countless discussions, reflexions and rewritings,
always leading to a better understanding, and I warmly thank him for his
uncommon involvement and patience.

\section*{Proof of the proposition}

\subsection{Overview} \label{ss:overview}

The following proof combines the strategy used by F. Sergeraert in \cite[Section 4]
{Se} with the method of approximation by conjugation introduced by D.~Anosov and
A.~Katok in \cite{AK} and later developped by many authors (see \cite{FK} and 
references therein). We start with a particular smooth vector field $\xi_0$ (the
same as in~\cite{Se}) and build $\xi$ as the limit of a sequence of deformations
$\xi_k$ where each $\xi_k$ is the pullback $h_k^* \xi_0$ of $\xi_0$ by a smooth 
diffeomorphism $h_k$. Thus, the flow $f_k^t$ of $\xi_k$ is related to the flow 
$f_0^t$ of $\xi_0$ by $f_k^t = h_k^{-1} \circ f_0^t \circ h_k$. The point is to 
cook up the conjugations $h_k$ so that the diffeomorphisms $f_k^t$, $k \ge 1$, 
converge in the $\CC^\infty$ topology for a dense set of times~$t$ but converge only
in the $\CC^1$ topology for some other time. In particular, the diffeomorphisms
$h_k$ diverge in the $\CC^2$ topology. Here, the behaviour of the initial vector 
field plays a crucial role: we take a vector field $\xi_0$ presenting plateaux 
which accumulate at the origin and whose heights tend to zero but with wild
oscillations. According to a theorem of F. Sergeraert \cite[Section 3]{Se}, these 
oscillations are necessary if we want to create a non-smooth flow with small 
perturbations $h_k$ of the identity. Furthermore, Theorem \ref{t:general} at 
the end of this paper states an oscillation condition which is sufficient for 
our construction to work. 

Let us indicate now how these oscillations come into play. First of all, we pick 
an initial vector field $\xi_0$ vanishing only at the origin, and contracting:
every point is attracted by $0$ in the future. Or, in other words, the function 
$\xi_0 / \partial_x$ is negative away from $0$. The graph of this function can
then be depicted as an undersea landscape consisting of a succession of alternating
lowlands $L_k$ and highlands $H_k$ whose respective altitudes $-v_k$ and $-u_k$ 
(measured from the water surface, 
so that $0 < u_k < v_k$) go to zero when $k$ grows, but ``oscillate wildly'' in 
the sense that the ratios $ v_k / u_k $ tend to infinity. 

\begin{figure}[htbp]
\psfrag{e}{$\scriptstyle H_k$}
\psfrag{a}{$\scriptstyle L_k$}
\psfrag{b}{$\scriptstyle H_{k+1}$}
\psfrag{c}{$\scriptstyle L_{k+1}$}
\psfrag{d}{$\scriptstyle H_{k+2}$}
\psfrag{f}{$\scriptstyle \xi_0$}
\centering
\includegraphics[height=1.5cm,width=13cm]{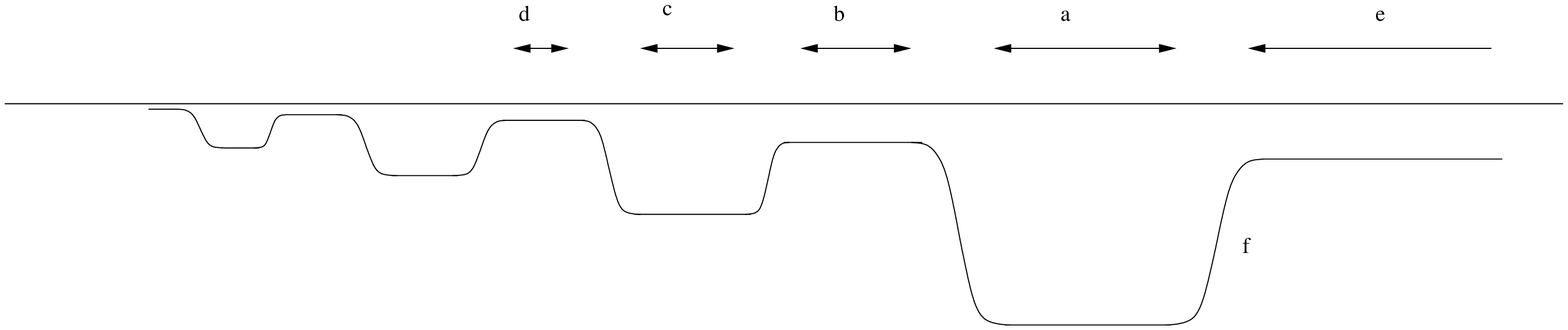}
\end{figure}

A consequence 
of this behaviour is that, if an element $f_0^t$ of the flow takes a segment $S 
\subset H_k$ into $L_k$ for some large $k$, then its restriction to $S$ is an 
affine map with big dilation factor $ v_k / u_k $.

In our deformation process, the diffeomorphisms $h_k$ are defined inductively 
and all coincide with the identity near $0$. Each new perturbation is described 
by the diffeomorphism $g_k = h_k \circ h_{k-1}^{-1}$ and its role is to modify 
the flow of $\xi_0$ locally at a specific time $1/{q_k}$, in a fundamental 
segment $S_k$ of $f_0^{1/{q_k}}$ lying in the lowland $L_k$. In other words,
$g_k^{-1} \circ f_0^{1/{q_k}} \circ g_k$ agrees with $f_0^{1/{q_k}}$ 
outside~$S_k$. Furthermore, we take $g_k$ close enough to the identity so that
the $\CC^k$ norms of the maps 
$$ g_k^{-1} \circ f_0^t \circ g_k - f_0^t, \qquad
   t \in \frac 1 {q_k} \Z \cap [0,1], $$
$$(\text{ and also } \quad h_k^{-1} \circ f_0^t \circ h_k - h_{k-1}^{-1} \circ f_0^t
\circ h_{k-1}\;,\;)$$
are all strictly bounded by $2^{-k}$, and we denote by $I_k$ a compact
neighbourhood of $\frac1{q_k} \Z
\cap (0,1)$ such that the non-strict bounds still hold for all $t \in I_k$. With 
a suitable choice of the sequence~$q_k$, we can arrange that the intersection of
the compact sets $I_k$ is a Cantor set $K$ consisting of irrational times~$t$ 
for which the diffeomorphisms $h_k^{-1} \circ f_0^t \circ h_k$, $k  \ge 1$, 
converge in the $\CC^\infty$ topology. Indeed, it suffices to pick $q_k$ at each
step in such a way that $\frac1{q_k} \Z$ meets any component of $I_{k-1}$ in at
least two points, and also avoids  the $k\th$ rational number (for an arbitrary
enumeration of $\Q$) so that $K = \cap I_k$ is totally irrational.

Although the action of the perturbation diffeomorphism $g_k$ on the map $f_0^{
1/{q_k} }$ is local, its action on the vector field $\xi_0$ and on general 
elements of its flow is not at all. To see this, let us consider the difference 
$\nu_k = g_k^* \xi_0 - \xi_0$. Since $g_k$ commutes with $f_0^{1/{q_k}}$ out
of~$S_k$ and coincides with the identity near $0$, it is actually the identity on 
the whole interval $[0, \min S_k]$. In particular, $\nu_k$ vanishes identically 
there. Inside $S_k$, our choice of $g_k$ gives $\nu_k$ the following shape of
a $\CC^k$-small wave:
 
\begin{figure}[htbp]
\centering
\includegraphics[height=1.5cm,width=12cm]{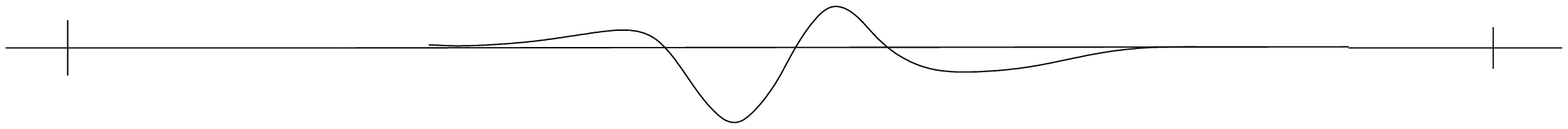}
\end{figure}

On the other hand, the half-line $[\max S_k, +\infty)$ is tiled by the segments 
$S_k^p = f_0^{- p/{q_k}} (S_k)$, $p \ge 1$. The commutation property noted
above now implies that, for every $p \ge 1$, 
\begin{equation} \label{e:propagation}
   \nu_{k\;|S_k^p } 
 = \left( f_0^{ p/{q_k}} \right)^* \left(\nu_{k \;| S_k }\right) . 
\end{equation}
In other words, the wave $\nu_{k \;| S_k }$ is propagated to the right of $S_k$
by the iterates of $f_0^{1/{q_k}}$. Let us look at the wave $\nu_{k \;| S_k^p }$ 
when $S_k^p$ sits on the highland $H_k$. As explained before, the
restriction of $f_0^{ p/{q_k}}$ to $S_k^p$ for such a $p$ is an affine map
of the form
$$ x \in S_k^p \mapsto \frac{ v_k }{ u_k }\; x + c_k \quad
   \text{for some} \quad c_k \in \R . $$   
Then, according to \eqref{e:propagation},
\begin{equation*}
   \left(\nu_{k \;\;| S_k^p }\right) (x)
 = \frac{ \left(\nu_{k \;| S_k }\right) \left( f_0^{ p/{q_k}} (x) \right) }
   { Df_0^{p/{q_k}} (x) } 
 = \left( \frac{ v_k }{ u_k } \right)^{-1} 
   \left(\nu_{k \;| S_k }\right) \left( \frac{ v_k }{ u_k }\; x + c_k \right),    
\end{equation*}
and so, for any integer $m \ge 1$,
$$ D^m \left(\nu_{k \;\;| S_k^p }\right) (x)
 = \left( \frac{ v_k }{ u_k } \right)^{m-1} 
   D^m (\nu_{k \;\;| S_k }) \left( \frac{ v_k }{ u_k }\; x + c_k \right) . $$   
Thus, in the course of the propagation, the wave remains $\CC^1$ small but its
higher order derivatives are amplified and can become big. As we already said, 
the difficulty is then to adjust the perturbation diffeomorphisms $g_k$ so that 
the differences $h_k^*\xi_0 - \xi_0$ (which are essentially the superpositions 
of the propagated waves $\nu_l$, $l \le k$) diverge in the $\CC^2$ topology while 
the conjugates $h_k^{-1} \circ f_0^t \circ h_k$, for $t$ in the Cantor set~$K$, 
still converge in the $\CC^\infty$ topology. Following Sergeraert, a solution is 
roughly to take $u_k$ and $v_k$ respectively equal to $2^{-k^4}$ and $2^{-k^2}$,
while the size of the wave $\nu_{k \;\;| S_k }$ is set to $2^{-k^3}$.

\subsection{Notations and toolbox} \label{ss:toolbox}

In this short section, we assume that all necessary conditions are met so that
the expressions we write make sense. For any $\CC^k$ map $g$ defined on an 
interval $I \subset \R$ (open or not), we set
\begin{equation*}
   \Bars{ g }_k
 = \sup \left\{ \bigbars{ D^mg (x) },\ 0 \le m \le k, \ x \in I \right\}. 
\end{equation*}
If $g \from I \to g(I)$ is an orientation-preserving $\CC^2$ diffeomorphism, we
define $Lf$ to be  
\begin{equation*}
   Lf = D \log Df = \frac{ D^2 f }{ Df } .
\end{equation*}
The non-linear differential operator $L$ satisfies the following chain rule:
\begin{equation} \label{e:chainrule}
   L (h \circ g) = Lh \circ g \cdot Dg + Lg . 
\end{equation}
To compute higher order derivatives of compositions, we will also use Faà di 
Bruno's formula in the form
\begin{equation} \label{e:faadibruno}
   D^m (h \circ g) = \sum_{\pi \in \Pi_m} 
   \left(D^{|\pi|} h\right) \circ g \cdot \prod_{B \in \pi} D^{|B|} g,
\end{equation}
where $\Pi_m$ is the set of all partitions $\pi$ of $\{ 1, \cdots, m \}$ and 
$|X|$, for any finite set $X$, is the number of its elements. 

Let $\eta$ be a vector field on an interval $J$. Throughout the paper, we will
make no difference between $\eta$ and the function $\eta / \partial_x$, where
$x$ is the underlying coordinate in $J$, and in particular we will identify
$\partial_x$ with $1$. If $J$ is both the source of $g$ and the target of $h$
(where $g$ and $h$ are diffeomorphisms), we can define two new vector fields,
$g_*\eta$ and $h^*\eta$, which are the pushforward of $\eta$ by $g$ and its
pullback by $h$, respectively. Viewed as functions, these vector fields are 
given by
\begin{align} \label{e:pushforward}
   g_*\eta 
&= Dg \circ g^{-1} \cdot \eta \circ g^{-1}, \\
   h^*\eta &= \frac{ \eta \circ h }{ Dh } \label{e:pullback}
\end{align}
and so we easily get the following expressions for the derivatives:
\begin{align} \label{e:Dpushforward}
   D (g_*\eta) 
&= D\eta \circ g^{-1} + Lg \circ g^{-1} \cdot \eta \circ g^{-1}, \\
   D (h^*\eta) 
&= D\eta \circ h - \frac{ D^2h }{ (Dh)^2 } \;\eta \circ h . \label{e:Dpullback} 
\end{align}

\subsection{The initial vector field} \label{ss:initialization}

The construction involves two smooth functions $\alpha, \beta \from \R \to 
[0,1]$ satisfying the following conditions: 
\begin{itemize}
\item
$\alpha(x)$ equals $0$ if $x \le 1/ 6$ and $1$ if $x \ge 1/ 3$; 
\item
$\beta(x)$ equals $0$ if $x \le 1/ 6$ or $x \ge  5/ 6$, and $1$ if 
$1/ 3 \le x \le 2/ 3 $;
\end{itemize}

\begin{figure}[htbp]
\psfrag{a}{$\scriptstyle 1$}
\psfrag{b}{$\scriptstyle 0$}
\psfrag{c}{$\scriptstyle \frac{1}{6}$}
\psfrag{d}{$\scriptstyle \frac{1}{3}$}
\psfrag{e}{$\scriptstyle \frac{2}{3}$}
\psfrag{f}{$\scriptstyle \frac{5}{6}$}
\psfrag{x}{$\scriptstyle \alpha$}
\psfrag{y}{$\scriptstyle \beta$}
\psfrag{z}{$\scriptstyle \gamma$}
\centering
\includegraphics[height=3cm,width=10cm]{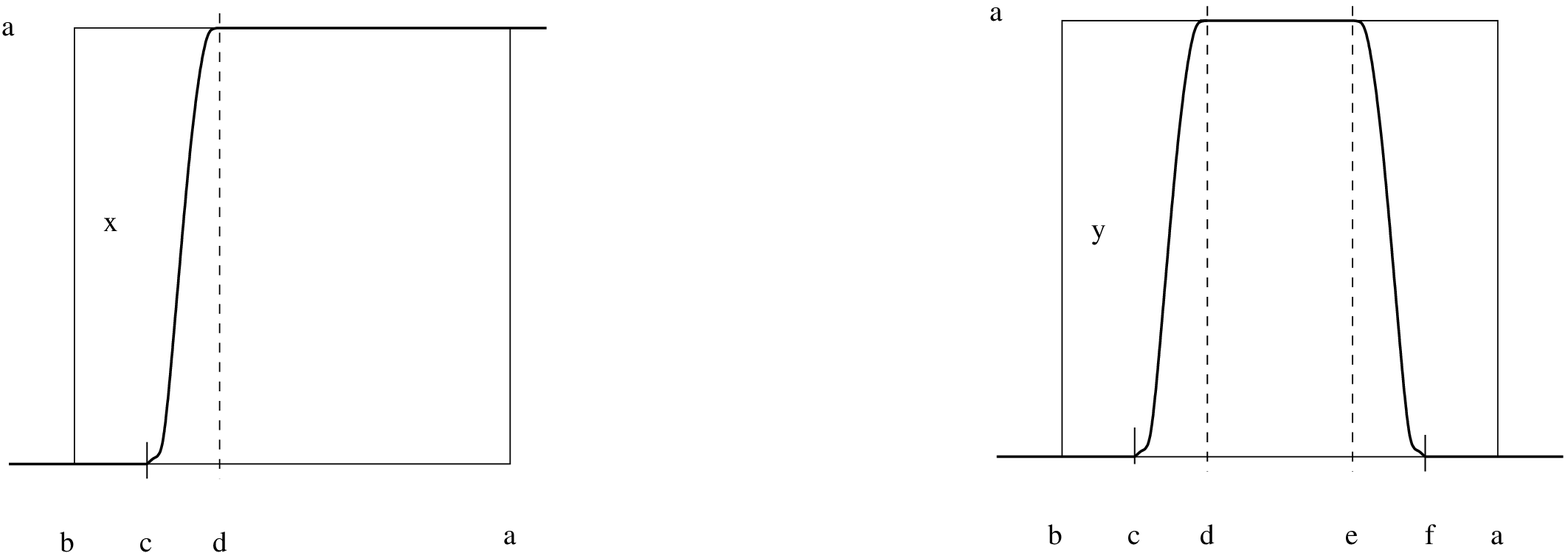}
\end{figure}

Now, setting $u_n = 2^{-n^4}$ and $v_n = 2^{-n^2}$, we define the vector field
$\xi_0$ as in \cite{Se} by
\begin{gather*}
   \xi_0(x) = - u_{n+1} - (u_n - u_{n+1})\; \alpha (2^{n+1} x - 1)
 - (v_n - u_n)\; \beta (2^{n+1} x - 1) \\ 
   \text{for $x \in [2^{-n-1}, 2^{-n}]$}, \quad 
   \xi_0(0) = 0 \quad \text{and} \quad
   \xi_0(x) = -1 \quad \text{for $x \ge 1$.} 
\end{gather*}

\begin{figure}[!h]
\psfrag{a}{$\scriptstyle 2^{-n-1}$}
\psfrag{c}{$\scriptstyle 2^{-n}$}
\psfrag{d}{$\scriptstyle u_n$}
\psfrag{e}{$\scriptstyle v_n$}
\psfrag{f}{$\scriptstyle u_{n+1}$}
\psfrag{g}{$\scriptstyle \xi_0$}
\centering
\includegraphics[width=10cm,height=1.5cm]{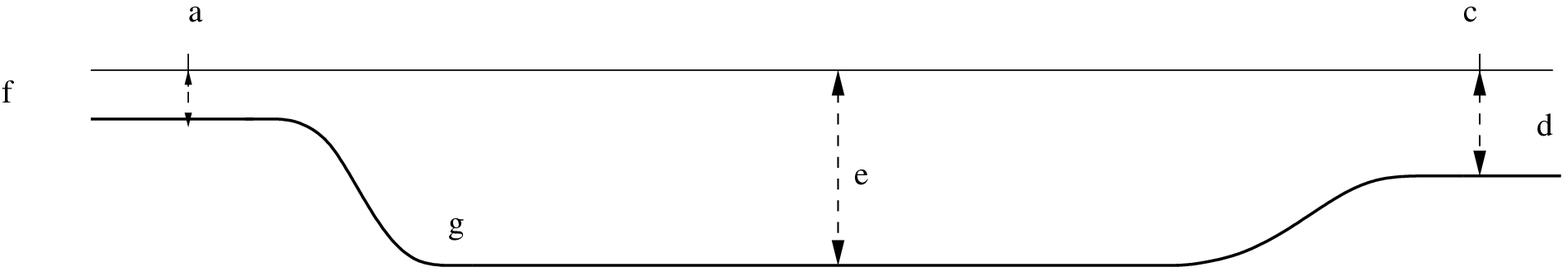}
\end{figure}

From now on, we denote by $\{ f_0^t, t \in \R \}$ the flow of $\xi_0$ and by
$\psi \from \R \to \R_+^*$ the diffeomorphism given by $\psi(t) = f_0^t(1)$ for all 
$t \in \R$. Note that, since $D\psi = \xi_0 \circ \psi$,
\begin{equation} \label{e:psi}
   \xi_0 = D\psi \circ \psi^{-1} \quad \text{and} \quad
   D\xi_0 = L\psi \circ \psi^{-1} .
\end{equation}
We also fix a forward orbit $\{ a_l, l \ge 0 \}$ of $f_0 = f_0^1$,
where $a_0 = 1$ and $a_l = f_0(a_{l-1})  = \psi(l)$ for all $l \ge 1$. 

One easily checks that $\xi_0$ is smooth, contracting, infinitely flat at the 
origin and $\CC^1$-bounded ---~with $1 < \Bars{\xi_0}_1 < +\infty$. Furthermore,
$\xi_0$ equals $-v_n$ identically on the central third of $[2^{-n-1}, 2^{-n}]$, 
namely $[2^{-n-1} + 2^{-n-1}/3, 2^{-n} - 2^{-n-1}/3]$, and $-u_n$ on $[2^{-n} - 
2^{-n-1}/6, 2^{-n} + 2^{-n}/6]$. A simple computation of travel time at constant
speed shows that for all $n \ge 4$, there exist integers $i(n)$ and $j(n)$ such 
that 
\begin{align} \label{e:ain}
2^{-n} - \frac 1{6} {2^{-n-1}} & \le a_{i(n)+2} < a_{i(n)-1} 
\le 2^{-n} + \frac 1{6}{2^{-n}} \\
\llap{\text{and} \quad}
2^{-n-1} + \frac  1{3}{2^{-n-1}} & \le a_{j(n)+2} < a_{j(n)-1} \le
2^{-n} - \frac  1{3}{2^{-n-1}} . \label{e:ajn}
\end{align}
Thus $\xi_0$ equals $-v_n$ on $[a_{ j(n)+2 }, a_{ j(n)-1 }]$,
and hence $f_0^t$ induces on $[a_{ j(n)+1 }, a_{ j(n)-1 }]$ the translation by 
$-tv_n$ for $0 \le t \le 1$. Similarly, $f_0^t$ induces the translation by $-t
u_n$ in a neighbourhood of $a_{ i(n) }$.

\subsection{The deformation process} \label{ss:deformation}

Our goal is now to produce a sequence $h_k$ of smooth diffeomorphisms of $\R_+$ 
such that the vector fields $\xi_k = h_k^* \xi_0$ converge in the $\CC^1$ topology 
to the vector field $\xi$ presented in Proposition~\ref{p:field}. In order to 
have regular perturbation patterns (and easier computations), we actually work 
at time scale, \emph{i.e.} we define $h_k$ as the conjugate $\psi \circ \Phi_k 
\circ \psi^{-1}$ of a smooth diffeomorphism $\Phi_k$ of~$\R$ (which coincides 
with the identity near $+\infty$ so that $h_k$ is also the identity near $0$). 
Conforming to the general scheme of the approximation by conjugation method (see
\cite{FK}), $\Phi_k$ is obtained as a composition 
$$ \Phi_k = \varphi_k \circ \varphi_{k-1} \circ \dots \circ \varphi_1 $$
where the diffeomorphisms $\varphi_k$ are manufactured inductively from a fixed 
function $\gamma$ and two adjustment integer parameters $q_k$ and $n_k$. Details
of the construction follow.

\begin{figure}[!h]
\psfrag{a}{$\scriptstyle -\frac 1 4$}
\psfrag{c}{$\scriptstyle -\frac 1 {20}$}
\psfrag{d}{$\scriptstyle \frac 1 {20}$}
\psfrag{f}{$\scriptstyle \frac 1 4$}
\psfrag{g}{$\scriptstyle 1$}
\psfrag{h}{$\scriptstyle \gamma$}
\centering
\includegraphics[width=5cm,height=3cm]{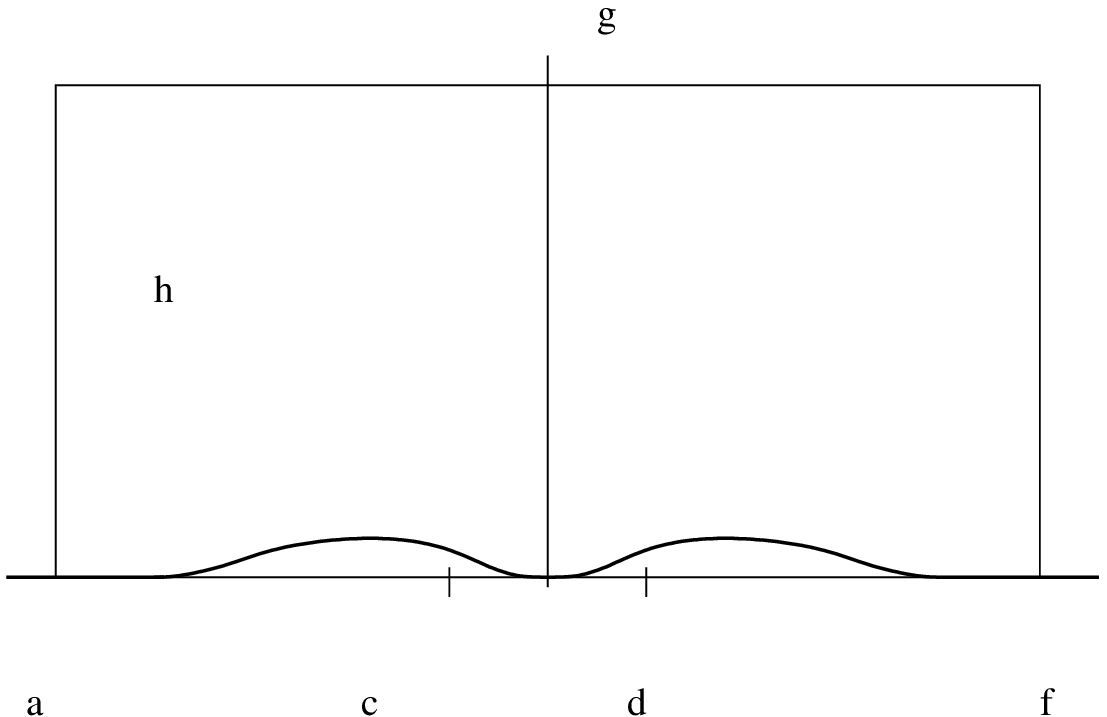}
\end{figure}

Let $\gamma \from \R \to [0,1]$ be a smooth function supported in $[-1/4, 
1/4]$ and satisfying $\gamma(t) = t^2/2$ around $0$. Given 
positive integers $q, n$, set $w_n = 2^{-n^3}$ and denote by $\gamma_{q,n} \from
\R \to [0,1]$ the smooth function defined by
\begin{equation} \label{e:gamma}
   \gamma_{q,n} (t) = w_n \gamma \Bigl( q \bigl( t-j(n) \bigr) \Bigr) \qquad
   \text{for all $t \in \R$.} 
\end{equation}
Clearly, $\gamma_{q,n}$ is supported in $\left[j(n)-\frac1{4q}, j(n)+\frac1{4q}\right]$.
Moreover, for every integer $m \ge 1$ and all $t \in \R$,
\begin{equation*} 
   D^m \gamma_{q,n} (t)
 = w_n q^m D^m \gamma \Bigl( q \bigl( t-j(n) \bigr) \Bigr)
\end{equation*}
and hence
\begin{equation*}
      \Bars{ \gamma_{q,n} }_m = w_nq^m \Bars{ \gamma }_m.
\end{equation*}
In particular, by taking $n$ large compared to $q$ once $m$ is fixed, one can 
make the $\CC^m$ norm of $\gamma_{q,n}$ arbitrarily small.

Now let $J_{q,n}$ be the interval $\left[j(n)-\frac1{2q}, j(n)+\frac1{2q}\right]$
and define $\varphi_{q,n} \from \R \to \R$ as the map meeting the following
properties:
\begin{itemize}
\item
$\varphi_{q,n} (t) = t$ for $t > j(n) + \frac1 {2q}$; 
\item
$\varphi_{q,n} (t) = t + \gamma_k(t)$ for $t \in J_{q,n}$;
\item
$\varphi_{q,n}$ commutes with the translation by $\frac 1q$ outside $J_{q,n}$,
and so $\varphi_{q,n} (t) = t + \gamma_k \left( t + \frac p q \right)$ if $t \in
\left(J_{q,n} - \frac p q\right)$, $p \ge 0$.
\end{itemize}
In short, we can write
\begin{align} \label{e:phi}
   \varphi_{q,n} (t) &= t + \sum_{p \ge 0} 
   \gamma_{q,n} \left( t + \frac p q \right), \\
\intertext{and similarly,}
   D\varphi_{q,n} (t) &= 1 + \sum_{p \ge 0} 
   D\gamma_{q,n} \left( t + \frac p q \right), \notag \\
   D^m\varphi_{q,n} (t) &= \sum_{p \ge 0} 
   D^m\gamma_{q,n} \left( t + \frac p q \right) \qquad 
   \text{for all $m \ge 2$.} \notag
\end{align}
Note that for every $t \in \R$, at most one term in each sum is nonzero since 
the support of $\gamma_{q,n}$ has length less than $1/ q$. These equations
imply that $\Bars{ \varphi_{q,n} - \id }_m = \Bars{ \gamma_{q,n} }_m$ and, in
particular, $\varphi_{q,n}$ is a diffeomorphism provided $\Bars{\gamma_{q,n}}_1 
< 1$.


\medskip

The following lemma will be used later (in the proof of Lemma \ref{l:limit}) to 
show that the limit flow coming out of our construction is not smooth at time~$
1/2$:

\begin{lemma} \label{l:Phi}
For all $k \ge 1$, let $q_k$ and $n_k$ be positive integers with $q_k$ odd and 
$w_{n_k} q_k \Bars{ \gamma }_1 < 1$. Then the diffeomorphism $\Phi_k$ defined by
\begin{equation*}
   \Phi_k = \varphi_k \circ \varphi_{k-1} \circ \dots \circ \varphi_1, \quad
   \text{where} \quad \varphi_k = \varphi_{q_k,n_k}, 
\end{equation*}
has the following behaviour on $\frac12 \Z$ for every $k \ge 1$\up:
\begin{itemize}
\item 
$\Phi_k$ coincides with the identity in a neighbourhood of $\Z + \frac12$\up;
\item
$\Phi_k$ is tangent to the identity on $\Z$ ---~meaning that $\Phi_k(l) = l$ and
$D\Phi_k(l) = 1$ for all $l \in \Z$\up;
\item 
$(L\Phi_k - L\Phi_{k-1}) (l)$, for $l \in \Z$, equals $w_{n_k} q_k^2$ if $l \le
j(n_k)$ and $0$ otherwise.
\end{itemize}
\end{lemma}

\begin{proof}
Since $\gamma_k = \gamma_{q_k,n_k}$ is supported in $\left[ -\frac1{4q_k}, 
\frac1{4q_k} \right] + j(n_k)$ and 
$$ \varphi_k = \id + \sum_{p \ge 0} \gamma_k \left( t + \frac p {q_k} \right), $$
$\varphi_k$ is the identity on the $\frac1{4q_k}$-neighbourhood of $\frac1{q_k} 
\Z + \frac1{2q_k}$. But $q_k$ is odd, $q_k = 2p+1$ say, so 
$$\frac12=\frac{ q_k }{2q_k} =\frac{ 2p+1 }{2q_k} 
= \frac p{q_k} + \frac 1{2q_k}\in\frac 1{q_k} \Z + 
\frac 1{2q_k}, $$ 
and hence $\frac1{q_k} \Z + 
\frac1{2q_k}$ contains $\Z + \frac12$. Therefore $\Phi_k$ is the identity in a 
neighbourhood of $\Z + \frac12$. On the other hand, since $\gamma(0) =
D\gamma(0) = 0$, each $\varphi_k$ is tangent 
to the identity on $ \frac 1{q_k}\Z \supset \Z$, so $\Phi_k$ is tangent to the 
identity on $\Z$.

Now, applying the chain rule \eqref{e:chainrule} for the operator $L = D^2/D$ to 
$\Phi_k = \varphi_k \circ \Phi_{k-1}$, we get
\begin{equation*}
   L\Phi_k = L\varphi_k \circ \Phi_{k-1} \cdot D\Phi_{k-1} + L\Phi_{k-1}  .
\end{equation*}
For $l \in \Z$, we have seen above that $\Phi_{k-1} (l) = l$ and $D\Phi_{k-1} (l)
= 1$, so
$$ (L\Phi_k - L\Phi_{k-1}) (l) = L\varphi_k(l) . $$ 
If $l > j(n_k)$ then $L\varphi_k (l) = 0$ just because $\varphi_k$ agrees with 
the identity on $\left[j(n_k) + \frac1{2q_k}, +\infty\right)$. If $l \le j(n_k)$
then \eqref{e:gamma} and \eqref{e:phi} give 
$$ L\varphi_k (l) = \frac{ D^2\varphi_k (l) }{ D\varphi_k (l) }
 = D^2\varphi_k (l) = D^2\gamma_k \bigl( j(n_k) \bigr) = w_{n_k} q_k^2, $$ 
which completes the proof. 
\end{proof}

For the next lemma, we fix an enumeration of rational numbers, $\Q = \{ r_k, k 
\ge 1 \}$, and set $\Phi_0 = \id$ and $I_0 = [0,1]$. Moreover, as in Lemma \ref
{l:Phi}, we will henceforth abbreviate $\varphi_{q_k,n_k}$ as $\varphi_k$ (and
similarly $\gamma_{q_k,n_k}$ as $\gamma_k$ and $J_{q_k,n_k}$ as $J_k$).
 
\def\esti{\mathrm{i}}
\def\estii{\mathrm{ii}}
\def\estiii{\mathrm{iii}}

\begin{lemma} \label{l:induction}
For suitably chosen increasing sequences of positive integers $q_k$ and~$n_k$, 
the diffeomorphisms $\Phi_k = \varphi_k \circ \dots \circ \varphi_1$ and $h_k = 
\psi \circ \Phi_k \circ \psi^{-1}$, the vector fields $\xi_k = h_k^* \xi_0$ and 
their flows $f_k^t$ satisfy the following estimates for every $k \ge 1$\up:
\begin{align*}
   &\bigBars{ \Phi_k - \Phi_{k-1} }_{k+1} \le 2^{-k-1}, 
    \tag{$\esti_k$} \label{1k} \\
   &\bigBars{ \xi_k - \xi_{k-1} }_1 \le 2^{-k}, 
    \tag{$\estii_k$} \label{2k} \\
   &\bigBars{ f_k^t - f_{k-1}^t }_k \le 2^{-k} \quad
   \text{for all} \quad t \in I_k \cup \{1\},\tag{$\estiii_k$} \label{3k}
\end{align*}
where $I_k \subset I_{k-1}$ is a compact set avoiding the $k\th$ rational number
$r_k$ and consisting of $2^k$ disjoint segments of nonzero length, two in each
component of $I_{k-1}$. 
\end{lemma}

\begin{proof}
Let $k \ge 1$ and assume we already chose $q_l$ and $n_l$ for $1 \le l \le k-1$ 
in such a way that estimates ($\esti_l$), ($\estii_l$) and ($\estiii_l$) hold. 
In particular, since $\Phi_0 = \id$ by convention,
\begin{equation} \label{e:Phik-1}
   \bigBars{ \Phi_{k-1} - \id }_2 \le
   \sum_{l=1}^{k-1} \bigBars{ \Phi_l - \Phi_{l-1} }_2 \le 
   \sum_{l=1}^{k-1} 2^{-l-1} = \frac12 - 2^{-k} \le \frac12 . 
\end{equation}
Take an odd integer $q_k > q_{k-1}$ such that $ \frac1{q_k} \Z$ avoids $r_k$ and
meets the interior of each component of $I_{k-1}$ in at least two points. Then 
pick $n_k > n_{k-1}$ such that
\begin{equation} \label{e:nk}
   \bigBars{ \gamma_k }_{k+1} \le
   \frac{ 2^{-k-4} \; v_{n_k}^{k-1} }{ |\Pi_{k+1}| \; 
   \bigBars{ D\Phi_{k-1} }_k^{k+1} \, \bigBars{ \xi_0 }_1 }, 
\end{equation}
\emph{i.e.}
\begin{equation*} 
   \frac{ w_{n_k} }{ v_{n_k}^{k-1} } \le
   \frac{ 2^{-k-4} \;\;q_k^{-k-1} }{ |\Pi_{k+1}| \; 
   \bigBars{ \gamma }_{k+1} \, \bigBars{ D\Phi_{k-1} }_k^{k+1} \, \bigBars{ \xi_0 }_1 },
\end{equation*}
which is possible since
$$ \frac{ w_n }{ v_n^{k-1} } = 2^{-n^3 + (k-1)n^2} = \o(1) . $$
Note that inequality \eqref{e:nk} clearly implies $\bigBars{\gamma_k}_1 < 1$, and 
so $\varphi_k$ is a diffeomorphism (remember $\bigBars{\varphi_k-\id}_m=\bigBars{\gamma_k}_m$).

\medskip

Let us first prove that this choice of $n_k$ implies \eqref{1k}. Since $\Phi_k =
\varphi_k \circ \Phi_{k-1}$, Faà di Bruno's formula \eqref{e:faadibruno} gives, 
for $0 \le m \le k+1$,
\begin{equation*}
   D^m (\Phi_k - \Phi_{k-1}) 
 = \sum_{\pi \in \Pi_m} D^{|\pi|} (\varphi_k - \id) \circ \Phi_{k-1} \cdot
   \prod_{B \in \pi} D^{|B|} \Phi_{k-1} .
\end{equation*}
But for every partition $\pi \in \Pi_m$ with $m \le k+1$, 
\begin{gather*}
   \bigBars{ D^{|\pi|} (\varphi_k - \id) \circ \Phi_{k-1} }_0 =
   \bigBars{ \gamma_k }_{|\pi|} \le \bigBars{ \gamma_k }_{k+1} \\
   \llap{\text{and}\quad}
   \prod_{B \in \pi} \bars{ D^{|B|} \Phi_{k-1} } \le
   \bigBars{ D\Phi_{k-1} }_k^{k+1}, 
\end{gather*}
and so
\begin{equation*}
   \bigBars{ \Phi_k - \Phi_{k-1} }_{k+1} \le |\Pi_{k+1}| \;\;
   \bigBars{ \gamma_k }_{k+1} \, \bigBars{ D\Phi_{k-1} }_{k}^{k+1} . 
\end{equation*}    
Thus, by the choice of $n_k$ in \eqref{e:nk},
\begin{equation*}
   \bigBars{ \Phi_k - \Phi_{k-1} }_{k+1} \le 
   \frac{ 2^{-k-4} \;\;v_{n_k}^{k-1} }{ \bigBars{ \xi_0 }_1 } \le 2^{-k-1}, 
\end{equation*}
which is the desired estimate \eqref{1k} (note that $\Bars{ \xi_0 }_1 \ge 1$).

\medskip

To prove \eqref{2k}, let us define 
\begin{equation*}
	\eta_k = \Phi_k^* \partial_t - \Phi_{k-1}^* \partial_t \quad \text{and} \quad 
   \zeta_k = \varphi_k^* \partial_t - \partial_t\;,
\end{equation*}
so that 
\begin{equation*}
   \eta_k = \Phi_{k-1}^* \zeta_k \quad \text{and} \quad
   \xi_k - \xi_{k-1} = \psi_* \eta_k\;.
\end{equation*}
Viewing $\zeta_k$ as a function, 
\begin{equation*}
   \zeta_k = \frac 1 { D\varphi_k } - 1 \quad \text{and} \quad 
   D\zeta_k = - \frac{ D^2 \varphi_k }{ (D\varphi_k)^2 } . 
\end{equation*}
Given the choice of $n_k$ in \eqref{e:nk}, 
\begin{align*} 
   &\bigBars{ D\varphi_k - 1 }_0 = \bigBars{ D\gamma_k }_0 \le 2^{-k-4} \;\bigBars{
\xi_0 }_1^{-1}\quad
\left(\text{ and so } \lrBars{\frac 1 {D\varphi_k}}_0\le 2\right),\\
  \text{and } &\bigBars{ D^2\varphi_k }_0 = \bigBars{ D^2\gamma_k }_0 \le 2^{-k-4} \;\Bars{ \xi_0 }_1^{-1}, 
   \end{align*}
so 
\begin{equation}
   \bars{ \zeta_k } \le 2^{-k-3} \;\Bars{ \xi_0 }_1^{-1} \quad \text{and} \quad 
   \bars{ D\zeta_k } \le 2^{-k-2}\; \Bars{ \xi_0 }_1^{-1} .
\label{e:zeta}
\end{equation}
Next, applying \eqref{e:pullback} and \eqref{e:Dpullback} to $\eta_k =
\Phi_{k-1}^* \zeta_k$, 
\begin{equation*}
   \eta_k = \frac{ \zeta_k \circ \Phi_{k-1} }{ D\Phi_{k-1} } \quad 
   \text{and}  \quad
   D\eta_k= D\zeta_k \circ \Phi_{k-1} 
 - \frac { D^2\Phi_{k-1} }{ (D\Phi_{k-1})^2 } \;\zeta_k \circ \Phi_{k-1}
\end{equation*}
so, according to \eqref{e:Phik-1} and \eqref{e:zeta}, 
\begin{align*}
   \bars{ \eta_k } &\le 2^{-k-2} \;\Bars{ \xi_0 }_1^{-1}, \\
   \bars{ D\eta_k } &\le
   2^{-k-2} \;\Bars{ \xi_0 }_1^{-1} 
 + \frac 4 2 \;2^{-k-3} \;\Bars{ \xi_0 }_1^{-1} \le 
   2^{-k-1} \;\Bars{ \xi_0 }_1^{-1} .
\end{align*}
Now, applying \eqref{e:pushforward}, \eqref{e:Dpushforward} and \eqref{e:psi} to $\xi_k - \xi_{k-1} = \psi_* \eta_k$, 
\begin{align*}
   \bars{ \xi_k - \xi_{k-1} } 
&= \bars{ \eta_k \circ \psi^{-1} \cdot \xi_0 } \le 
   \Bars{ \eta_k }_0 \, \Bars{ \xi_0 }_0 \le 2^{-k-2}, \\
   \bars{ D(\xi_k - \xi_{k-1}) }
&= \bars{ D\eta_k \circ \psi^{-1} + D\xi_0 \cdot \eta_k \circ \psi^{-1} }
   \\ &\le
   2 \Bars{ \eta_k }_1 \, \Bars{ \xi_0 }_1 \le 2^{-k} .
\end{align*}
Thus, $\Bars{ \xi_k - \xi_{k-1} }_1 \le 2^{-k}$ as stated in estimate \eqref
{2k}.

\medskip

Let us finally prove \eqref{3k}. Set $\varphi_0 = \id$ and denote by 
$\sigma_l^t$ the flow of $\varphi_l^* \partial_t$ for $0 \le l \le k$. Then
$\sigma_0^t$ is just the translation by $t$ and 
\begin{equation*}
   \sigma_k^t = \varphi_k^{-1} \circ \sigma_0^t \circ \varphi_k . 
\end{equation*}
Since 
\begin{align*}
   \xi_k &= \psi_* \Phi_k^* \partial_t
 = \psi_* \Phi_{k-1}^* \varphi_k^* \partial_t \qquad & \text{and} \qquad 
   \xi_{k-1} &= \psi_* \Phi_{k-1}^* \partial_t, \\
\intertext{their flows are given by}
   f_k^t 
&= \psi \circ \Phi_{k-1}^{-1} \circ \sigma_k^t \circ \Phi_{k-1} \circ \psi^{-1} 
   \quad & \text{and} \quad 
   f_{k-1}^t 
&= \psi \circ \Phi_{k-1}^{-1} \circ \sigma_0^t \circ \Phi_{k-1} \circ \psi^{-1}.
\end{align*}
By definition, $\varphi_k = \varphi_{q_k,n_k}$ commutes with the translation 
$\sigma_0^{1/{q_k}}$ outside $J_k = J_{q_k,n_k}$. Consequently, $\varphi_k$ 
commutes with any iterate $\sigma_0^{p/{q_k}}$, $p \ge 1$, outside the 
interval
$$ \left[ j(n_k) + \frac1{2q_k} - \frac p{q_k}, j(n_k) + \frac1{2q_k} \right]
 = \bigcup_{q=0}^{p-1} \left(J_k - \frac q{q_k}\right) . $$
Therefore, $\sigma_k^{p/{q_k}}$ equals $\sigma_0^{p/{q_k}}$ outside 
this interval, and in particular, for $0 \le p \le q_k$, outside
\begin{equation} \label{e:Mk}
   M_k 
 = \left[ j(n_k) - 1 + \frac 1 {2q_k}, j(n_k) + \frac 1 {2q_k} \right] .
\end{equation}
On the other hand, for $t \in J_k$, 
\begin{align*}
   \sigma_k^{1/{q_k}} (t) 
&= \varphi_k^{-1} \left( \varphi_k(t) + \frac1{q_k} \right) &&\\
&= \varphi_k^{-1} \left( t + \gamma_k(t) + \frac1{q_k} \right) &\quad&
   \text{by definition of $\varphi_k$ on $J_k$} \\
&= t + \frac1{q_k} + \gamma_k(t) &\quad&
   \text{because $t + \gamma_k(t) + \frac1{q_k} > j(n_k) + \frac1{2q_k}$,} \\
&= \sigma_0^{1/{q_k}} (t) + \gamma_k(t). &&
\end{align*}
Thus, $\sigma_k^{1/{q_k}} - \sigma_0^{1/{q_k}} = \gamma_k$. Similarly, 
for any $p \ge 1$,
\begin{equation} \label{e:sigma}
   \sigma_k^{p/ {q_k}} (t) - \sigma_0^{ p/ {q_k} } (t)
 = \sum_{q=0}^{p-1} \gamma_k \left( t + \frac q {q_k} \right) \qquad
   \text{for all $t \in \R$,}
\end{equation}
so 
\begin{equation*}
   \lrBars{ \sigma_k^{ p/ {q_k} } - \sigma_0^{ p/ {q_k} } }_m
 = \bigBars{ \gamma_k }_m .
\end{equation*}
Now, in the region $M_k$ where $\sigma_k^{p/{q_k}}$ and $\sigma_0^{p/ 
{q_k}}$ disagree for $0 \le p \le q_k$, the diffeomorphism $\Phi_{k-1}$ is the
identity. Moreover, $\psi \left( j(n_k) \right) = a_{j(n_k)}$ and $\psi(M_k)
\subset \left[ a_{j(n_k)+1}, a_{j(n_k)-1} \right]$ so, by \eqref{e:ajn}, $\psi$ 
restricted to $M_k$ is an affine map with slope $-v_{n_k}$. As a consequence,
the derivatives of 
\begin{equation*}
   f_k^{p/{q_k}} = \psi \circ \Phi_{k-1}^{-1} \circ 
   \sigma_k^{p/{q_k}} \circ \Phi_{k-1} \circ \psi^{-1} 
\end{equation*}
have a simple expression on $\psi(M_k)$: 
\begin{equation*}
   D^m \left(f_k^{p/ {q_k}}\right) 
 = v_{n_k}^{1-m}\; D^m \left(\sigma_k^{p/ {q_k} }\right) \circ \psi^{-1} .
\end{equation*}
Similarly, again on $\psi(M_k)$,
\begin{equation*}
   D^m \left(f_{k-1}^{p/ {q_k} }\right)
 = v_{n_k}^{1-m}\; D^m \left(\sigma_0^{p/ {q_k} }\right) \circ \psi^{-1} .
\end{equation*}
Therefore, for $0 \le p \le q_k$ and $0 \le m \le k$,
\begin{equation*}
   \lrbars{ D^m \left( f_k^{p/{q_k}} - f_{k-1}^{p/{q_k}} \right) }
  \le v_{n_k}^{1-m} 
   \lrBars{ \sigma_k^{p/ {q_k} } - \sigma_0^{p/ {q_k} } }_m
 = v_{n_k}^{1-m} \bigBars{ \gamma_k }_m \le v_{n_k}^{1-k} \bigBars{ \gamma_k }_k \le
   2^{-k-4} 
\end{equation*}
according to our choice of $n_k$ in \eqref{e:nk}, and thus
\begin{equation*} \label{e:Dmflot}
   \BigBars{ f_k^t  - f_{k-1}^t } \le 2^{-k-4} \quad \text{for all} \quad
   t \in \frac 1 {q_k} \Z \cap [0,1] .
\end{equation*}
Now let $T_k$ be a subset of $\frac1 {q_k} \Z \cap I_{k-1}$ with exactly two 
points in each of the $2^{k-1}$ connected components of $I_{k-1}$ (remember that 
$q_k$ was chosen so that there are at least two points there). Since both vector
fields $\xi_k$ and $\xi_{k-1}$ are smooth, we can find a compact neighbourhood 
$I_k$ of $T_k$ in $I_{k-1} \setminus \{r_k\}$ consisting of $2^k$ segments and
such that
\begin{equation*} 
	\BigBars{ f_k^t  - f_{k-1}^t } \le 2^{-k} \quad \text{for all} \quad
   t \in I_k \cup \{1\} .
\end{equation*}
This completes the proof of \eqref{3k}, and thus of Lemma~\ref{l:induction}.
\end{proof}

\subsection{The limit vector field} \label{ss:limit}

\begin{lemma} \label{l:limit}
The vector fields $\xi_k$, $k \ge 1$, of Lemma \up{\ref{l:induction}} converge 
in the $\CC^1$ topology on $\R_+$, and in the $\CC^\infty$ topology on $\R_+^*$, to a 
vector field $\xi$ which satisfies all properties stated in Proposition \up{\ref
{p:field}} with $K = \cap\; I_k$.
\end{lemma}

\begin{proof}
The $\CC^1$ convergence of the vector fields $\xi_k$ on $\R_+$ follows directly 
from estimate \eqref{2k} in Lemma \ref{l:induction}. Next, estimate \eqref{1k} 
shows that the diffeomorphisms $\Phi_k$ converge in the $\CC^\infty$ topology to a 
smooth diffeomorphism $\Phi$ of $\R$, so the vector fields $\Phi_k^* \partial_t$
converge in the $\CC^\infty$ topology to $\Phi^* \partial_t$. Now $\xi_k$ equals 
$\psi_* \Phi_k^* \partial_t$ on $\R_+^*$ and $\psi$ is a smooth diffeomorphism 
$\R \to \R_+^*$. Given any compact set $A \subset \R_+^*$ and any integer $m \ge
0$, the restriction of $\psi$ to $\psi^{-1} (A)$ is  $\CC^m$-bounded, and hence 
the vector fields $\xi_k$ converge $\CC^m$ uniformly to $\xi$ on $A$. Therefore,
the vector fields $\xi_k$ converge to $\xi$ on $\R_+^*$ in the $\CC^\infty$ 
(compact-open) topology. 

\medskip

The convergence of the vector fields $\xi_k$ implies a similar convergence of 
their flows $f_k^t$ to the flow $f^t$ of $\xi$. Furthermore, estimate \eqref{3k} 
in Lemma \ref{l:induction} shows that, for $t \in  K \cup \{1\}$, the diffeomorphisms
$f_k^t$ converge in the $\CC^m$ topology on $\R_+$ for any $m \ge 0$. As a result,
they converge in the $\CC^\infty$ topology and $f^t$ is smooth for all $t \in K \cup
\{1\}$, and hence for all $t \in \Z \oplus \sum_{\tau \in K} \tau\Z$. Note here 
that each $I_k$, by construction, is a compact set avoiding the $k\th$ rational 
number and consisting of $2^k$ segments, two in each component of $I_{k-1}$, 
so $K = \cap \;I_k$ is indeed a Cantor set.

\medskip

The last thing we have to prove is that $f^{1/2}$ is not $\CC^2$ at $0$ or, 
equivalently, that $Lf^{1/2} = D^2f^{1/2} / Df^{1/2}$ is not 
continuous at $0$. Let us compute $Lf^{1/2}$ at a point $a_{ i(n_l) } $, as
defined in \eqref{e:ain}, for $l \in \N$. Taking the limit of the maps
\begin{equation*}
   f_k^{1/2 } = \psi \circ \Phi_k^{-1} \circ \left(\id + \frac12\right) \circ
   \Phi_k \circ \psi^{-1},
\end{equation*}
we get 
\begin{equation*}
   f^{ 1/2 }
 = \psi \circ \Phi^{-1} \circ \left(\id + \frac12\right) \circ \Phi \circ \psi^{-1} .
\end{equation*}
Let us set $\sigma = \Phi^{-1} \circ (\id + \frac12) \circ \Phi$, so that 
$f^{1/2} = \psi \circ \sigma \circ \psi^{-1}$. Near $a_{ i(n_l) }$, the map 
$\psi^{-1}$ is affine, with slope $-u_{n_l}^{-1}$, so 
\begin{equation*}
   Lf^{1/2 } \left( a_{ i(n_l) } \right) = - \frac 1 { u_{n_l} } L\sigma \bigl( i(n_l) \bigr) . 
\end{equation*}
On the other hand, by \eqref{e:chainrule} applied twice,
\begin{equation*}
   L\sigma \bigl( i(n_l) \bigr)
 = L \Phi^{-1} \left( \Phi \bigl( i(n_l) \bigr) + \frac12 \right) \cdot 
   D\Phi \bigl( i(n_l) \bigr) 
 + L\Phi \bigl( i(n_l) \bigr) .
\end{equation*}
According to Lemma \ref{l:Phi}, each $\Phi_k$, and hence $\Phi$, is tangent to
the identity on $\frac12 \Z$ provided all integers $q_k$ were chosen odd.
Moreover, $\Phi_k$ and $\Phi_k^{-1}$ coincide with the identity near $\Z + 
\frac12$, so $L \Phi^{-1} \left( i(n_l) + \frac12 \right) = 0$. Summing up, and using
the third property in Lemma \ref{l:Phi}, we get
\begin{equation} \label{e:Lsigma}
   L\sigma \bigl( i(n_l) \bigr)
 = L\Phi \bigl( i(n_l) \bigr)
 = \sum_{k \ge 1} (L\Phi_k - L\Phi_{k-1}) \bigl( i(n_l) \bigr)
 = \sum_{k \ge l} w_{n_k} q_k^2.
\end{equation}
In the end,
\begin{equation*}
   Lf^{ 1/2 } \left( a_{ i(n_l) } \right) 
 = - \frac 1 { u_{n_l} } \sum_{k \ge l} w_{n_k} q_k^2
 < - \frac{ w_{n_l} }{ u_{n_l} } \to - \infty,
\end{equation*}
and so $f^{1/2}$ is not $\CC^2$ at $0$.
\end{proof}

\section*{More examples}

Let $\SS$ denote the space of smooth diffeomorphisms of $\R_+$ which are 
infinitely tangent to the identity at the origin and have no other fixed point.
We say that a diffeomorphism $f$ of $\R_+$ is contracting if $f(x) < x$ for all 
$x>0$, and we call \emph{Szekeres vector field} of $f$ the unique $\CC^1$ vector
field generating the one-parameter group $\ZZ^1_f$ \cite{Sz, Ko}.

As mentioned in the introduction, the question we discuss in this section is 
whether the phenomenon presented in Theorem \ref{t:centralizer} is very
peculiar or quite general. First of all, because of Takens' work \cite{Ta}, this
phenomenon is limited to $\SS$. A difficulty then is that there is no obviously 
relevant topology on $\SS$ for our problem. In particular, the $\CC^\infty$ 
compact-open topology restricted to $\SS$ is extremely coarse: given any two 
diffeomorphisms $f, g \in \SS$, which are both contracting, say, it is easy to 
construct a sequence of diffeomorphisms $f_k \in \SS$ which converge to $f$ in 
the $\CC^\infty$ topology and whose germs at $0$ are all equal to that of~$g$. In
other words, the $\CC^\infty$ topology does not see the germ at $0$ while this 
germ precisely determines the smoothness of the Szekeres vector field and hence 
the nature of the centralizers in the groups $\DD^r$ for $r \ge 2$. So we do not
claim that the phenomenon described in Theorem~\ref{t:centralizer} is generic in
any way, but the following result shows that it is at least not scarce:

\begin{theorem} \label{t:general}
Let $f_0$ be a smooth contracting diffeomorphism of $\R_+$ having a smooth and 
$\CC^1$-bounded Szekeres vector field, and satisfying the following oscillation 
condition\up:
\begin{equation} \label{e:oscillation}
	\limsup_{ x \to 0 } \left( \sup_{0<y\le x}\frac
   { \bigbars{ \log \bigl( x - f_0(x) \bigr) } }
   { \bigbars{ \log \bigl( y - f_0 (y) \bigr) } } \right) 
 = +\infty.
\end{equation}
Then, for every $k \ge 0$ and every $\varepsilon > 0$, there exists a smooth 
diffeomorphism $f$ of $\R_+$ which is close to $f_0$ in the sense that
\begin{equation} \label{e:closeness}
   \bigbars{ D^m (f-f_0) (x) } \le \varepsilon \bigbars{ D^m (f_0-\id) (x) }
   \qquad \text{for all $m \le k$ and all $x \in \R_+$,}
\end{equation}
and whose centralizer $\ZZ^\infty_f$ is a proper, dense and uncountable subgroup
of $\ZZ^1_f$.
\end{theorem}

Note that the oscillation condition \eqref{e:oscillation} forces $f_0$ to be 
infinitely tangent to the identity at~$0$. 

\medskip
 
It is interesting to compare this result with Theorem 3.1 in \cite{Se}. Indeed, 
the latter says that, if a smooth contracting diffeomorphism $f$ does not 
oscillate much in the sense that  
\begin{equation*}
   \sup_{0 < y \le x } \bigl( y - f(y) \bigr)
 = \O \left( \bigl( x - f(x) \bigr)^\lambda \right) \qquad
   \text{for some } \lambda > \frac{ r-1 } r,
\end{equation*}
then the Szekeres vector field of $f$ is $\CC^r$. Theorem \ref{t:general} can be
thought of as a kind of ``partial converse''. 

\begin{proof}
The idea of the proof is the same as for Theorem \ref{t:centralizer}: we start 
with a smooth vector field, here the Szekeres vector field $\xi_0$ of the given 
$f_0$ instead of Sergeraert's vector field, and construct deformations $\xi_k$ 
of $\xi_0$ which converge to the Szekeres vector field $\xi$ of the wanted~$f$. 
We will just hint at how to adapt the arguments in this more general setting. As
before, we denote by $f_0^t$ the flow of $\xi_0$ (so that $f_0 = f_0^1$) and by
$\psi$ the diffeomorphism $\R \to \R_+^*$ given by $\psi(t) = f_0^t(1)$ for all
$t \in \R$. We also fix a forward orbit of $f_0$, namely $\{ a_l = f_0^l(1) =
\psi(l), l \ge 0 \}$, and we set $V_l = [a_{l+2}, a_{l-2}]$ for all $l \ge 0$.

\begin{lemma} \label{l:orbit}
There exist two alternating sequences of integers $i(n)$ and $j(n)$, $n \ge 0$, 
with $i(n) < j(n) < i(n+1) < j(n+1) < \ldots$, such that
\begin{equation} \label{e:divlog}
   \frac{ \log u_n }{ \log v_n } \xrightarrow[ n \to \infty ]{} +\infty
\end{equation}
where $u_n = \sup_{ V_{i(n)} } |\xi_0|$ and $v_n = \inf_{ V_{j(n)} } |\xi_0|$.
In particular, $V_{i(n)}$ and $V_{j(n)}$ are disjoint when $n$ is large enough.
\end{lemma}

The proof of Lemma \ref{l:orbit} is left to the reader. It relies on the 
oscillation property \eqref{e:oscillation} of $f_0$ and the fact that $\xi_0$ is
$\CC^1$. 

\medskip

We now choose a sequence $w_n$ with intermediate decay, \emph{i.e.} satisfying 
$w_n = \o(v_n^m)$ for all $m$ and $u_n = \o(w_n)$ (one can take for instance
$w_n=\sqrt{u_n}$). Then we define the maps $\gamma_{q,n}$ and $\varphi_{q,n}$
by formulae \eqref{e:gamma} and \eqref{e:phi}, using the same function $\gamma$ 
but the new parameters $w_n$ and $j(n)$. Extending thence all other definitions 
and notations of Subsection \ref{ss:deformation}, our task is to show that 
Lemmas \ref{l:induction} and \ref{l:limit} still hold.

\begin{proof}[Proof of Lemma \ref{l:induction} in the general setting]
We only insist here on the points that differ from the proof in subsection \ref
{ss:deformation}. Again, we proceed by induction. At step $k$, the choice of 
$q_k$ is just the same, but we need to be more careful about $n_k$. The reason
is that the map $\psi$ is no longer affine on the regions we consider, and hence
the computation of higher derivatives of compositions is trickier.
 
First, using the fact that $\xi_0$ is smooth and infinitely flat at $0$, one can
check that, for any fixed $m \ge 1$, 
\begin{gather*} 
   \sup  \left\{\bars{ D^m\psi (t)},\,t\in \left[j(n)-1, \infty\right)\right\}
   \xrightarrow[ n \to \infty ]{} 0 \\ 
   \llap{\text{and}\quad}
   v_n^{m+1} \sup \left\{\bars{ D^m\psi^{-1} (x)},\, x\in \left[a_{ j(n)+1 },
a_{ j(n)-1 } \right] \right\}
   \xrightarrow[ n \to \infty ]{} 0 . 
\end{gather*}
(this is derived from the relations $D\psi = \xi_0 \circ \psi$ and $D\psi^{-1}
= 1/\xi_0$).

Then we pick an integer $n_k > n_{k-1}$ meeting the following three conditions:
\begin{align}
	\BigBars{ D\psi_{ \;\;\big|\left[j(n_k)-1, \infty\right) }}_{k-1} &< 1, \label{e:normpsi} \\
   \BigBars{ D^m\psi^{-1} _{ \quad\big|\left[a_{j(n_k)+1}, a_{j(n_k)-1}\right] } }_0 
   &< v_{n_k}^{-m-1} \quad \text{for $1 \le m \le k$,} \label{e:Dmpsiinvk} \\
   \llap{\text{and}\quad}
   \Bars{ \gamma_k }_{k+1} &\le \frac{ 2^{-k^2-4} \;v_{n_k}^{2k} }
   { |\Pi_{k+1}|^2  \Bars{ D\Phi_{k-1} }_k^{k+1} \, \Bars{ \xi_0 }_1 }. 
   \label{e:nkbis}
\end{align}
Inequality \eqref{e:nkbis} is stronger than \eqref{e:nk} and thus implies \eqref
{1k} and \eqref{2k} of Lemma \ref{l:induction} (the arguments are strictly the 
same). The proof of \eqref{3k} is more complicated but we still have (with our
former notations)
\begin{equation*}
   f_k^t
 = \psi \circ \Phi_{k-1}^{-1} \circ \sigma_k^t \circ \Phi_{k-1} \circ \psi^{-1} 
   \quad \text{and}  \quad 
   f_{k-1}^t
 = \psi \circ \Phi_{k-1}^{-1} \circ \sigma_0^t \circ \Phi_{k-1} \circ \psi^{-1}.
\end{equation*}
For $t = p /q_k$, $0 \le p \le q_k$, again $\sigma_k^t = \sigma_0^t$
outside $M_k = \left[ j(n_k) - 1 + \frac1{2q_k}, j(n_k) + \frac1{2q_k} \right]$, 
so $f_k^t - f_{k-1}^t = 0$ outside $\psi (M_k)$. Furthermore, $\Phi_{k-1} = \id$ on 
$M_k$. Thus, on $\psi(M_k)$,
\begin{equation*}
   f_k^t = \psi \circ \sigma_k^t \circ \psi^{-1} \quad \text{and} \quad 
   f_{k-1}^t = \psi \circ \sigma_0^t \circ \psi^{-1}
\end{equation*}
or, equivalently,
\begin{equation*}
   f_k^t - f_{k-1}^t = (\psi \circ \sigma_k^t) \circ \psi^{-1}
 - (\psi \circ \sigma_0^t) \circ \psi^{-1} .
\end{equation*}
For $m \le k$, Faà di Bruno's formula gives
\begin{equation} \label{e:Dmflotb}
   D^m \bigl(f_k^t - f_{k-1}^t\bigr) = \sum_{ \pi \in \Pi_m } 
   D^{ |\pi| } \bigl(\psi \circ \sigma_k^t - \psi \circ \sigma_0^t\bigr) \circ \psi^{-1}
   \cdot \prod_{B \in \pi } D^{|B|} \psi^{-1}.
\end{equation}
According to inequality \eqref{e:Dmpsiinvk}, 
\begin{equation} \label{e:prod}
   \lrbars{ \prod_{ B \in \pi } D^{|B|} \psi^{-1} } < v_{n_k}^{-2k} \quad 
   \text{on} \quad 
   \psi(M_k) \subset \left[ a_{ j(n_k)+1 }, a_{ j(n_k)-1 } \right] .
\end{equation}
Now write
\begin{equation*}
   \psi \circ \sigma_k^t - \psi \circ \sigma_0^t
 = (\psi \circ \sigma_0^t) \circ (\sigma_0^{-t} \circ \sigma_k^t)
 - (\psi \circ \sigma_0^t)
\end{equation*}
and observe, using \eqref{e:sigma}, that 
\begin{equation*}
   \sigma_0^{-t} \circ \sigma_k^t = \id + \sum_{q=0}^{p-1}
   \gamma_k \left( \id + \frac q {q_k} \right) \qquad
   \text{for $t = \frac p {q_k}$, $0 \le p \le q_k$.}
\end{equation*}
For $l \le k$, Faà di Bruno's formula gives 
\begin{align*}
   D^l \left(\psi \circ \sigma_k^t - \psi \circ \sigma_0^t\right) 
&= D^l \left(\left(\psi \circ \sigma_0^t\right) \circ \left(\sigma_0^{-t} \circ \sigma_k^t\right)
 - \left(\psi \circ \sigma_0^t\right)\right) \\
&= \sum_{ \pi \in \Pi_l,\;|\pi|<l } 
   D^{ |\pi| } \left(\psi \circ \sigma_0^t\right) \circ \left(\sigma_0^{-t} \circ \sigma_k^t\right) 
   \cdot \prod_{B \in \pi } D^{|B|}\left( \sigma_0^{-t} \circ \sigma_k^t\right) .
\end{align*}
Since $\sigma_0^t = \id + t$, it follows from \eqref{e:normpsi} that 
\begin{equation*}
   \bigbars{ D^{ |\pi| } (\psi \circ \sigma_0^t) \circ 
   (\sigma_0^{-t} \circ \sigma_k^t) }
 < 1 \quad \text{on } M_k. 
\end{equation*}
On the other hand, for any partition $\pi \in \Pi_l$ with less than $l$ blocks, 
\emph{i.e.} $|\pi| < l$, one block $B$ of $\pi$ has at least two elements, so at
least one factor in the product 
\begin{equation*}
   \prod_{B \in \pi} D^{|B|} \left(\sigma_0^{-t} \circ \sigma_k^t\right)
 = \prod_{B \in \pi} D^{|B|} \left(\id + \sum_{q=0}^{p-1}
   \gamma_k \left( \id + \frac q {q_k} \right)\right)
\end{equation*} 
is a derivative of order at least $2$, and hence is bounded above by $\Bars
{\gamma_k}_k$, while the others are all less than $2$. Therefore the product is 
bounded above by $2^{l-2} \Bars{ \gamma_k }_k \le 2^{k-2} \Bars{ \gamma_k }_k$. 
In the end, 
\begin{equation*}
   \Bigbars{ D^l (\psi \circ \sigma_k^t - \psi \circ \sigma_0^t) } \le
   |\Pi_l| \, 2^{l-2} \Bars{ \gamma_k }_k \le 
   |\Pi_k| \, 2^{k-2} \Bars{ \gamma_k }_k .
\end{equation*}
In view of \eqref{e:Dmflotb}, \eqref{e:prod} and \eqref{e:nkbis} this implies 
that $\bigBars{ f_k^t - f_{k-1}^t}_k \le 2^{-k-4}$ for all $t =  p /{q_k}$, 
$0 \le p \le q_k$, and one completes the proof of Lemma \ref{l:induction} just 
as in Subsection \ref{ss:deformation}.
\end{proof}

\begin{proof}[Proof of Lemma \ref{l:limit} in the general setting]
The proof that the vector fields $\xi_k$ converge and that the limit flow $f^t$
is smooth for $t \in \Z \oplus \sum_{\tau \in K} \tau \Z$ is strictly the same
as in Subsection \ref{ss:limit}. Note that if we start our construction at step 
$k_0$ instead of step $1$, the limit diffeomorphism $f$ satisfies the condition
\eqref{e:closeness} for $l \le k_0$ and $\varepsilon = 2^{-k_0-1}$, so one can 
construct $f$ arbitrarilly close to $f_0$ in the sense of Theorem \ref
{t:general}. 

The part of Lemma \ref{l:limit} that needs a little extra effort is the 
irregularity of $f^{1/2}$. Again,
\begin{equation*}
   f^{1/2} = \psi \circ \sigma \circ \psi^{-1}, 
\end{equation*}
with $\sigma = \Phi^{-1} \circ (\id + 1/2) \circ \Phi$. The computation of
$L\sigma \bigl( i(n_l) \bigr)$ leading to \eqref{e:Lsigma} can be integrally 
transposed here, and yields $L\sigma \bigl( i(n_l) \bigr) = \sum_{k \ge l} 
w_{n_k} q_k^2$ (with the new $w_n$). This time however, $\psi$ is not affine on 
the involved region, so the computation of $Lf^{1/2} \bigl( i(n_l) \bigr)$ 
is a bit longer. Formula \eqref{e:chainrule} applied twice gives
\begin{equation*}
   Lf^{ 1/2 }
 = \Big( L\psi \circ (\sigma \circ \psi^{-1}) \cdot
   D (\sigma \circ \psi^{-1}) \Big)
 + \Big( L\sigma \circ \psi^{-1} \cdot D\psi^{-1} \Big)
 + L\psi^{-1},
\end{equation*}
and hence, since $D\psi^{-1} =  1 /{ \xi_0 }$,
\begin{equation*}
   Lf^{ 1/2 } \bigl( a_{ i(n_l) } \bigr)
 = \left[ L\psi \circ (\sigma \circ \psi^{-1}) \cdot D (\sigma \circ \psi^{-1}) 
 + L\psi^{-1} \right] \bigl( a_{ i(n_l) } \bigr) 
 + \frac{ L\sigma \bigl( i(n_l) \bigr) }{ \xi_0 \bigl( a_{i(n_l)} \bigr) } .
\end{equation*}
Now, according to Lemma \ref{l:Phi} (still valid in our new setting), the limit 
$\Phi$ of the diffeomorphisms $\Phi_k$ coincides with the translation by $
1/2$ at order one on $\Z$, so the first term of the above sum is equal to
\begin{equation*}
   \left[ L\psi \circ \left(\id + \frac12\right) \circ \psi^{-1} \cdot 
  D \left( \left(\id + \frac12\right) \circ \psi^{-1} \right) + L\psi^{-1}
\right]
   \bigl( a_{i(n_l)} \bigr) 
 = Lf_0^{ 1/2 } \bigl( a_{ i(n_l) } \bigr) .
\end{equation*}
But when $l$ grows, $Lf_0^{1/2} \bigl( a_{i(n_l)} \bigr)$ tends to $Lf_0^{
1/2} (0) = 0$. Therefore
\begin{equation*}
   Lf^{ 1/2 } \bigl( a_{i(n_l)} \bigr) \sim \frac 
   {\sum_{ k \ge l } w_{n_k} q_k^2}{\xi_0 \bigl( a_{i(n_l)} \bigr)} < - \frac{w_{n_l}}{ u_{n_l}} 
   \xrightarrow[ l \to \infty ]{} - \infty,
\end{equation*}
so $f^{ 1/2 }$ is not $\CC^2$ at $0$. This concludes the proof of Lemma
\ref{l:limit} and of Theorem \ref{t:general}.
\end{proof}
\def\qedsymbol{}
\end{proof}

\end{document}